\documentclass[12pt]{amsart}
\usepackage{amsthm,amsfonts,amsmath,amssymb,amsaddr,geometry,eufrak,bigints,eucal}
\usepackage{graphicx}
\input xypic
\usepackage{url}
\newtheorem{Theorem}{Theorem}[section]

\newtheorem{Corollary}{Corollary}[section]

\newtheorem{Definition}{Definition}[section]
\newtheorem{Example}{Example}[section]

\newtheorem{Remark}{Remark}[section]

\numberwithin{equation}{section}

\begin{document}

\title{Some Fixed Point Theorems in $(\alpha,\beta)$- Metric Spaces with applications to Fredholm integral and non-linear differential equations}
\author{Irfan Ahmed$^1$, Shallu Sharma$^2$, Sahil Billawria$^{3,*}$}
\maketitle

\begin{center}
\footnotesize $^1$Department of Mathematics, University of Jammu, Jammu, India. Email: shallujamwal09@gmail.com\\
$^2$Department of Mathematics, University of Jammu, Jammu, India. Email: irfan.ahmed@jammuuniversity.ac.in\\
$^3$ Department of Mathematics, University of Jammu, Jammu, India. Email: sahilbillawria2@gmail.com\\
$^*$Corresponding author: sahilbillawria2@gmail.com\\
\end{center}
\begin{abstract}
In this paper, we presented a new type of metric space called $(\alpha,\beta)$-metric space along with some novel contraction mappings named $(\alpha,\beta)$-contraction  and weak $(\alpha,\beta)$-contraction mapping. We established some fixed point theorem for these newly introduced contractive mappings. Our results extended some fixed point results in the existing literature. We also provide an example which holds for the weak $(\alpha,\beta)$-contraction. Furthermore, we proved Kannan's fixed point theorem and Reich's fixed point theorem in the setting of $(\alpha,\beta)$-metric spaces. At the end, as applications the Fredholm integral and non-linear differential equations are solved in order to validate the theoretically obtained conclusions.
\end{abstract}

\section{introduction}\label{sec1}
Metric spaces plays an important role in solving various mathematical problems. A variety of metric spaces paved a way to learn optimization and approximation theory, variational inequalities and many other problems. The idea of metric spaces was first given by Frechet in 1906 \cite{41}, a French mathematician. After this, many mathematician incorporated various generalization/extension of metric spaces in various ways. In 1931, Wilson \cite{43} introduced a metric space excluding symmetry condition, termed as quasi-metric space. The idea of $b$-metric space was given by Backhtin \cite{1} in 1989. Czerwik \cite{2,3} extended the results of $b$-metric spaces. Using this idea many researcher presented generalization of fixed point theorem in the b-metric space (see \cite{4,5,6,7}). In 2017, Zada et al. \cite{9} studied fixed point theorems in $b$-metric spaces and presented their applications to non-linear fractional differential and integral equations. The study of fixed point results in b-metric spaces provides a rich and demanding setting for mathematicians. It is highly useful in determining the existence and uniqueness of solutions to some ordinary differential equations with initial value conditions \cite{24,28,29,30}. Other generalizations include  $D$-metric space \cite{11,12}, Partial metric space \cite{13,222}, Cone Metric Space \cite{14,15,333}, Complex-Valued Metric Space \cite{16}, cone $b$-metric space \cite{17}, $S$-metric space \cite{18,19}, metric-like-space \cite{20,77}, extended-b-metric space \cite{22} and many more.\\
\\
 Banach's contraraction principle in metric spaces is one of the most fundamental results in fixed point theory and non-linear analysis, which states that if $(X,\mathfrak{D})$ is a complete metric space and $\Psi : X\rightarrow X$ is a contractive mapping i.e., there exists $\lambda\in[0,1)$ such that for all $\mathfrak{x,y}\in X$,
\begin{align*}
\mathfrak{D}(\Psi\mathfrak{x}, \Psi\mathfrak{y})\le\lambda\mathfrak{D(x,y)}.
\end{align*}
Then $\Psi$ has a unique fixed point. It is useful in solving a variety of problems pertaining to pure and applied mathematics. It offers a technique for assessing the fixed points of nonlinear equations and ensures their existence and uniqueness. Since 1922 \cite{25}, when Stefan Banach formulated the notion of contraction and proved the famous theorem, several authors have obtained various extensions of the Banach contraction principle by considering new contractive conditions and proved the various fixed point results (see \cite{111,55,66,88}). Recently, Vishnukumar et al. \cite{42} and Paul et al. \cite{99} studied some fixed point theorems in complex valued metric spaces and bipolar metric spaces. Every mapping of a Banach contraction is known to be a continuous function. In this sequel, it was reasonable to ask whether there exists any contraction map associated with a fixed point that is not necessarily continuous. In 1968, Kannan \cite{26}was the first mathematician who solved the answer and presented the theorem \ref{20}. Numerous fixed point results have been generalized in different metric spaces by using Kannan contraction mapping. Subrahmanyam \cite{36}, proved that Kannan's theorem characterises metric completeness, indicating that a metric space $(X,\mathfrak{D})$is complete if every Kannan mapping on $X$ has a fixed point. The Kannan's fixed point finding has been further extended and generalized in the setting of different variants of metric spaces (See \cite{31,32,33,34}. In 1971, Reich \cite{38} further extended the Banach and Kannan fixed point theorems (see \ref{21}). In the few decades various contraction mappings has been generalized in the sense of Reich (see \cite{39,40}). The study of Reich-type contractions has resulted in the development of several fixed point theorems, which are essential tools for establishing the existence and uniqueness of solutions to equations and inequalities in a wide range of mathematical and scientific domains. Integral and differential equations have significant applications in many fields of science and engineering including optimization, radiative heat transfer problems, mathematical economics, computer algorithms (see \cite{01,02,03}. Numerous fascinating and significant phenomena that are seen in many fields of science and technology have been mathematically modelled through the extensive application of non-linear differential equations. They are inspired by problems which arise in diverse fields such as economics, fluid mechanics, physics, differential geometry, engineering, control theory, material science and quantum mechanics (see \cite{04,05,06}).
Our achievements of this paper are listed as follows:
\begin{enumerate}
\item Introduced a new generalized metric space called $(\alpha,\beta)$-metric space and by providing an example we show that every $(\alpha,\beta)$-metric space need not be a metric space. 
\item Some new contractive mappings called $(\alpha,\beta)$-contraction and  weak $(\alpha,\beta)$-contraction are proposed and utilize the same to proved some fixed point results in $(\alpha,\beta)$-metric spaces. Also, we illustrate an example which holds for the weak $(\alpha,\beta)$-contraction.
\item Generlized the Kannan's fixed point theorem and Reich's fixed point theorem in the setting of $(\alpha,\beta)$-metric spaces.
\item We use our novel contraction mapping to verify that the Fredholm integral and non-linear differential equations have a unique solution.
\end{enumerate}
The $(\alpha,\beta)$-metric spaces is a proper generalization of both the metric spaces and $b$-metric spaces. By adding two independent scaling variables, it provides more flexibility in defining distances which leads to the generalization of many fixed point results in the existing literature.
Our paper is organized in a way begin with Section \ref{sec1}, as an introduction providing the background and context of this study followed by the preliminaries Section \ref{sec2}, which deals with some basic definitions and theorems which are used in our work. In Section \ref{sec3}, we developed our main results provided with some definitions, theorems, corollaries, remarks and example. In Section \ref{sec4} of this paper, we give the application of our work and at the end we conclude our paper in Section \ref{sec5}. 
\section{Preliminaries}\label{sec2}
This section deals with some basic definitions and theorems that are already existing in the literature which we utilize in our main results.
\begin{Definition}\cite{41}
Let $X$ be a non-empty set. A function $\mathfrak{G}: X\times X\rightarrow [0,\infty)$ is called a metric, if it satisfies the following properties for all $\mathfrak{x,y,z}\in X$:
\begin{enumerate}
\item $\mathfrak{G(x,y)}\ge 0 ~\forall ~\mathfrak{x,y}\in X.$
\item$\mathfrak{G(x,y)}= 0$ iff $\mathfrak{x=y}$.
\item$\mathfrak{G(x,y)= G(y,x)}~\forall~\mathfrak{x,y}\in X.$
\item$\mathfrak{G(x,y)\le G(x,z)+ G(z,y)},\forall~\mathfrak{x,y,z}\in X.$
\end{enumerate}
Then the pair $(X,\mathfrak{G})$ is called metric space.
\end{Definition}
\begin{Definition}\cite{1}
Let $X$ be a non-empty set. A function $\mathfrak{G}: X\times X\rightarrow [0,\infty)$ is called a  $b$-metric, if it satisfies the following properties for all $\mathfrak{x,y,z}\in X$:
\begin{enumerate}
\item $\mathfrak{G(x,y)}\ge 0 ~\forall ~\mathfrak{x,y}\in X.$
\item$\mathfrak{G(x,y)}= 0$ iff $\mathfrak{x=y}$.
\item$\mathfrak{G(x,y)= G(y,x)}~\forall~\mathfrak{x,y}\in X.$
\item$\mathfrak{G(x,y)\le S\Bigl[G(x,z)+G(z,y)}\Bigr]$ for $\mathfrak{S}\ge 1.$
\end{enumerate}
\end{Definition}

\begin{Definition}\cite{23}
Let $X$ be a non-empty set. A function $\mathfrak{G}: X\times X\rightarrow [0,\infty)$ is called a strong $b$-metric, if it satisfies the following properties for all $\mathfrak{x,y,z}\in X$:
\begin{enumerate}
\item $\mathfrak{G(x,y)}\ge 0 ~\forall ~\mathfrak{x,y}\in X.$
\item$\mathfrak{G(x,y)}= 0$ iff $\mathfrak{x=y}$.
\item$\mathfrak{G(x,y)= G(y,x)}~\forall~\mathfrak{x,y}\in X.$
\item$\mathfrak{G(x,y)\le G(x,z)}+\beta \mathfrak{G(z,y)}$ for $\beta\ge 1.$
\end{enumerate}
\end{Definition}
 A Strong $b$-metric space is a particular case of $b$-metric space.
\begin{Theorem}\label{20}
Let $(X,\mathfrak{D})$ be a complete metric space and $\Psi : X\rightarrow X$ be a mapping satisfying:
\begin{align*}
\mathfrak{D}(\Psi\mathfrak{x}, \Psi\mathfrak{y})\le\lambda\Bigl[\mathfrak{D(x},\Psi \mathfrak{x})+\mathfrak{D(y},\Psi \mathfrak{y})\Bigr] for~all~\mathfrak{x,y}\in X~and~0\le\lambda<\frac{1}{2}.
\end{align*}
Then $\Psi$ has a unique fixed point.
\end{Theorem}
\begin{Theorem}
$\textbf{(Reich's Fixed Point Theorem)}$\label{21} \cite{38}\\
Let $(X,\mathfrak{G})$ be a complete metric space. Let $\Psi:X\rightarrow X$ be a self map that satisfies the following contractive condition, i.e;  there exist constants $\xi_{1},\xi_{2},\xi_{3}\ge 0$ with $\xi_{1}+\xi_{2}+\xi_{3}<1,$ such that:
\begin{align*}
\mathfrak{G}(\Psi\mathfrak{x}, \Psi \mathfrak{y})\le \xi_{1}\mathfrak{G(x,y)}+\xi_{2}\mathfrak{G(x},\Psi \mathfrak{x})+\xi_{3}\mathfrak{G(y,}\Psi\mathfrak{y)}~\forall~\mathfrak{x,y}\in X.
\end{align*}
 Then $\Psi$ has a unique fixed point.
\end{Theorem}

\section{Main results}\label{sec3}
In this section, we introduced a novel generalized metric space, which we call as an $(\alpha,\beta)$-metric space with some new contractive mapping such as $(\alpha,\beta)$-contraction, weak $(\alpha,\beta)$-contraction mapping and by utilising same we proved some fixed point theorems and established some corollary and remarks. we also give an example to determine the fixed point by using the weak $(\alpha,\beta)$-contraction. Furthermore, we Generalized the Kannan's fixed point theorem and Reich's fixed point theorem in the setting of $(\alpha,\beta)$-metric spaces.
\begin{Definition}
A non-empty set $X$ that has a distance function $\mathfrak{G}: X\times X\rightarrow [0,\infty)$, called $(\alpha,\beta)$- metric, that meets the following conditions $\forall$ $\mathfrak{x,y,z}\in X$ and for some $\alpha,\beta\ge1:$
\begin{enumerate}
\item $\mathfrak{G(x,y)}\ge 0 ~\forall ~\mathfrak{x,y}\in X$.
\item$\mathfrak{G(x,y)}= 0$ iff $\mathfrak{x=y}$.
\item$\mathfrak{G(x,y)= G(y,x)}~\forall~\mathfrak{x,y}\in X$.
\item$\mathfrak{G(x,y)}\le \alpha \mathfrak{G(x,z)}+\beta \mathfrak{G(z,y)}$, for some  $\alpha,\beta\ge1$.
\end{enumerate}
The pair $(X,\mathfrak{G})$ denotes the $(\alpha,\beta)$- metric space.
\end{Definition}
Here, we have two independent parameters as opposed to $b$-metric spaces, where the triangle inequality is controlled by a single constant $\mathfrak{S}$. This allows more flexibility in applications where distances behave asymmetrically, such as in directed networks or certain function spaces.
\begin{Remark}
\begin{enumerate}
\item If $\alpha= \beta=1$, then it reduces to classical metric, and $(\alpha,\beta)$- metric space is a weaker notion.
\item If $\alpha= 1~and ~\beta=\mathfrak{S}$, then it reduces to the strong b-metric space.
\item if $\alpha=\beta=\mathfrak{S}$, then $\mathfrak{G(x,y)\le SG(x,z)+SG(z,y)=S\big(G(x,z)+G(z,y)\big)}$. Thus, every $b$-metric space is therefore a particular instance of a $(\alpha, \beta)$-metric space whenever $\alpha=\beta$.
\item Every metric space is an $(\alpha,\beta)$- metric, but the converse need not be true by the following example:
\end{enumerate}
\end{Remark}
\begin{Example}
Let $X=\mathbb{R}$ and define a function $\mathfrak{G(x,y)= |x-y|^{2}}$. If we choose $\mathfrak{x}=0,\mathfrak{y}=2$ and $\mathfrak{z}=1$, then $\mathfrak{G(x,y)= |x-y|^{2}}$ is not a metric, but it forms an $(\alpha,\beta)$- metric for $\alpha,\beta>2$.
\end{Example}
\begin{Definition}
Let $(X,\mathfrak{G})$ be an $(\alpha,\beta)$- metric space. Then a sequence $\{\mathfrak{x_{n}}\}\subseteq X$ is said to be convergent to a point $\mathfrak{x}\in X$ if and only if 
\begin{align*}
\lim_{n\to\infty}\mathfrak{G(x_{n},x)}=0.
\end{align*}
\end{Definition}
\begin{Definition}
Let $(X,\mathfrak{G})$ be an $(\alpha,\beta)$- metric space. Then a sequence $\{\mathfrak{x_{n}}\}\subseteq X$ is said to be cauchy if and only if
\begin{align*}
 \lim_{n,m\to\infty}\mathfrak{G(x_{n},x_{m})}=0.
 \end{align*}
or equivalently, for every $\epsilon>0$, there exist $\mathfrak{N_{\circ}}\in\mathbb{N}$ such that for all $\mathfrak{m,n}>\mathfrak{N_{\circ}}$, we have
\begin{align*}
\mathfrak{G(x_{n},x_{m})}<\epsilon.
\end{align*}

\end{Definition}
\begin{Definition}
Let $(X,\mathfrak{G})$ be an $(\alpha,\beta)$- metric space. Then $X$ is said to be complete $(\alpha,\beta)$-metric space if every cauchy sequence in $X$ is convergent in $X.$

\end{Definition}
\begin{Definition}
Let $(X,\mathfrak{G})$ be an $(\alpha,\beta)$- metric space. Then a mapping $\Psi:X\rightarrow X$ is said to be $(\alpha,\beta)$-contraction, if there exist constants $\xi_{1},\xi_{2}\ge 0$ with $\xi_{1}+\xi_{2}<1$, satisfies the following condition:\\
\begin{center}
$\mathfrak{G}(\Psi\mathfrak{x}, \Psi \mathfrak{y})\le \xi_{1}\mathfrak{G(x,y)}+\xi_{2}[\mathfrak{G(x},\Psi \mathfrak{x})+\mathfrak{G(y,}\Psi\mathfrak{y)}]~\forall~\mathfrak{x,y}\in X$.
\end{center}
\end{Definition}
\begin{Theorem}\label{11}
Let $(X,\mathfrak{G})$ be a complete $(\alpha,\beta)$- metric space. Let $\Psi:X\rightarrow X$ be an $(\alpha,\beta)$-contraction, i.e;  there exist constants $\xi_{1},\xi_{2}\ge 0$ such that $\xi_{1}+\xi_{2}<1$ and\\
\begin{equation}\label{1}
\mathfrak{G}(\Psi\mathfrak{x}, \Psi \mathfrak{y})\le \xi_{1}\mathfrak{G(x,y)}+\xi_{2}[\mathfrak{G(x},\Psi \mathfrak{x})+\mathfrak{G(y,}\Psi\mathfrak{y)}]~\forall~\mathfrak{x,y}\in X.
\end{equation}
 Then $\Psi$ has a unique fixed point.
\end{Theorem}
\begin{proof}
Let $\mathfrak{x_{\circ}}$ be an arbitrary point in $X$. Define a sequence $\{\mathfrak{x_{n}}\}$ in $X$ by the recursion\\
\begin{equation}\label{2}
 \mathfrak{x_{n+1}}=\Psi\mathfrak{x_{n}}, \forall~n=0,1,2,3...
 \end{equation}
Using  (\ref{1}) and (\ref{2}), we have
\begin{align*}
\mathfrak{G(x_{n+1}, x_{n+2}})&=\mathfrak{G}(\Psi\mathfrak{x_{n}}, \Psi \mathfrak{x_{n+1}})\\
\\
&\le\xi_{1}\mathfrak{G(x_{n},x_{n+1})}+\xi_{2}[\mathfrak{G(x_{n}},\Psi \mathfrak{x_{n}})+\mathfrak{G(x_{n+1},}\Psi\mathfrak{x_{n+1})}]\\
\\
&\le\xi_{1}\mathfrak{G(x_{n},x_{n+1})}+\xi_{2}[\mathfrak{G(x_{n}},\mathfrak{x_{n+1}})+\mathfrak{G(x_{n+1},} \mathfrak{x_{n+2})}]\\
\\
&=\xi_{1}\mathfrak{G(x_{n},x_{n+1})}+\xi_{2}\mathfrak{G(x_{n}},\mathfrak{x_{n+1}})+\xi_{2}\mathfrak{G(x_{n+1},} \mathfrak{x_{n+2})}\\
\\
\mathfrak{G(x_{n+1}, x_{n+2}})(1-\xi_{2})&\le(\xi_{1}+\xi_{2})\mathfrak{G(x_{n},x_{n+1})}.
\end{align*}
Since $\xi_{1}+\xi_{2}<1$, we have $1-\xi_{2}>0.$\\
\\
Therefore, we have
\begin{align*}
\mathfrak{G(x_{n+1}, x_{n+2}})\le \Bigl(\frac{\xi_{1}+\xi_{2}}{1-\xi_{2}}\Bigr)\mathfrak{G(x_{n},x_{n+1})}.
\end{align*}
Put $\Bigl(\frac{\xi_{1}+\xi_{2}}{1-\xi_{2}}\Bigr)=K$. Since, $\xi_{1}+\xi_{2}<1$, we have $0<K<1.$\\
\\
Thus,
\begin{align*}
 \mathfrak{G(x_{n}, x_{n+1}})&\le K \mathfrak{G(x_{n-1},x_{n})}\\
 \\
  &\le K^{2} \mathfrak{G(x_{n-2},x_{n-1})}\\
 & \vdots\\
   \mathfrak{G(x_{n}},x_{n+1})&\le K^{n} \mathfrak{G(x_{\circ},x_{1})}.
   \end{align*}
   Now, we apply the triangle inequality in $(\alpha,\beta)$- metric space to estimate $\mathfrak{G(x_{n}, x_{m}})$ for $m>n,$ $m,n\in\mathbb{N}.$
   \begin{align*}
    \mathfrak{G(x_{n}, x_{m}})&\le\alpha \mathfrak{G(x_{n}, x_{n+1}})+\beta \mathfrak{G(x_{n+1}, x_{m}})\\
    \\
    &\le\alpha \mathfrak{G(x_{n}, x_{n+1}})+\beta[\alpha \mathfrak{G(x_{n+1}, x_{n+2}})+\beta \mathfrak{G(x_{n+2}, x_{m}})]\\
    \\
     &=\alpha \mathfrak{G(x_{n}, x_{n+1}})+\alpha\beta\mathfrak{G(x_{n+1}, x_{n+2}})+\beta^{2} \mathfrak{G(x_{n+2}, x_{m}})\\
     \\
     &\le\alpha \mathfrak{G(x_{n}, x_{n+1}})+\alpha\beta\mathfrak{G(x_{n+1}, x_{n+2}})+\beta^{2}[\alpha \mathfrak{G(x_{n+2}, x_{n+3}})+\beta \mathfrak{G(x_{n+3}, x_{m}})]\\
      \\
      &\le\alpha \mathfrak{G(x_{n}, x_{n+1}})+\alpha\beta\mathfrak{G(x_{n+1}, x_{n+2}})+\alpha\beta^{2} \mathfrak{G(x_{n+2}, x_{n+3}})+\cdots\\
      \\
      &\le\alpha K^{n} \mathfrak{G(x_{\circ},x_{1})}+\alpha\beta K^{n+1} \mathfrak{G(x_{\circ},x_{1})}+\alpha\beta^{2} K^{n+2} \mathfrak{G(x_{\circ},x_{1})}+\cdots\\
      \\
      &=\alpha K^{n} \mathfrak{G(x_{\circ},x_{1})}\Bigl[1+\beta K+(\beta K)^{2}+\cdots\Bigr]\\
      \\
      &= \frac {\alpha K^{n}}{1-\beta K} \mathfrak{G(x_{\circ},x_{1}}).
      \end{align*}
      Therefore, we have
      \begin{align*}
     \mathfrak{G(x_{n}, x_{m}}) \le \frac {\alpha K^{n}}{1-\beta K} \mathfrak{G(x_{\circ},x_{1}}).
      \end{align*}
      Taking $\lim_{n,m\to\infty}$, Since $K<1, \lim_{n\to\infty}\frac{\alpha K^{n}}{1-\beta K} \mathfrak{G(x_{\circ},x_{1})}=0.$
 \\
Thus, $\lim_{n,m\to\infty}\mathfrak{G(x_{n}, x_{m}})=0.$\\
This shows that $\{\mathfrak{x_{n}}\}$ is a cauchy sequence in $X$. Since $X$ is Complete $(\alpha, \beta)$-metric space, so there exist $\mathfrak{x^{*}}\in X$ such that $\mathfrak{x_{n}}\rightarrow\mathfrak{x^{*}}$.\\
Now, we will show that $\mathfrak{x^{*}}$ is a fixed point of $\Psi$.\\
\begin{align*}
\mathfrak{G(x^{*}},\Psi\mathfrak{x^{*}})&\le\alpha \mathfrak{G(x^{*},x_{n+1}})+\beta\mathfrak{G(x_{n+1}},\Psi \mathfrak{x^{*}})\\
\\
&\le\alpha \mathfrak{G(x^{*},x_{n+1}})+\beta\mathfrak{G(}\Psi\mathfrak{ x_{n}},\Psi\mathfrak{x^{*}})\\
\\
&\le\alpha \mathfrak{G(x^{*},x_{n+1}})+\beta\Bigl[\xi_{1}\mathfrak{G(x_{n},x^{*}})+\xi_{2}[\mathfrak{G(x_{n}},\Psi \mathfrak{x_{n}})+\mathfrak{G(x^{*}},\Psi\mathfrak{x^{*})}]\Bigr]\\
\\
&=\alpha \mathfrak{G(x^{*},x_{n+1}})+\beta\xi_{1}\mathfrak{G(x_{n},x^{*}})+\beta\xi_{2}\mathfrak{G(x_{n}},\Psi \mathfrak{x_{n}})+\beta\xi_{2}\mathfrak{G(x^{*}},
\Psi\mathfrak{x^{*})}\\
\\
(1-\beta\xi_{2})(\mathfrak{G(x^{*}},\Psi\mathfrak{x^{*}})&\le\alpha \mathfrak{G(x^{*},x_{n+1}})+\beta\xi_{1}\mathfrak{G(x_{n},x^{*}})+\beta\xi_{2}[\alpha \mathfrak{G(x_{n},x^{*}})+\beta\mathfrak{G(x^{*},x_{n+1}})]\\
\\
&=\alpha \mathfrak{G(x^{*},x_{n+1}})+\beta\xi_{1}\mathfrak{G(x_{n},x^{*}})+\alpha\beta\xi_{2} \mathfrak{G(x_{n},x^{*}})+\beta^{2}\xi_{2}\mathfrak{G(x^{*},x_{n+1}})\\
\\
&=(\alpha+\beta^{2}\xi_{2})\mathfrak{G(x^{*},x_{n+1}})+\beta(\xi_{1}+\alpha\xi_{2})\mathfrak{G(x_{n},x^{*}}).
\end{align*}
Therefore, we have\\
\begin{align*}
\mathfrak{G(x^{*}},\Psi\mathfrak{x^{*}})\le\frac{(\alpha+\beta^{2}\xi_{2})}{(1-\beta\xi_{2})}\mathfrak{G(x^{*},x_{n+1}})+\frac{\beta(\xi_{1}+\alpha\xi_{2})}{(1-\beta\xi_{2})}\mathfrak{G(x_{n},x^{*}}).\\
\end{align*}
Taking $\lim_{n\to\infty}$, we get\\
\\
$\lim_{n\to\infty}\mathfrak{G(x^{*}},\Psi\mathfrak{x^{*}})=0.$\\
$\Rightarrow \mathfrak{x^{*}}=\Psi\mathfrak{x^{*}}.$\\
Therefore, $\mathfrak{x^{*}}$ is a limit point of $\Psi$.\\
\\
$\textbf{Uniqueness of Fixed Point:}$\\
\\
Suppose  $\mathfrak{y^{*}}\in X$ such that $\Psi(\mathfrak{y^{*})=y^{*}}$. Then, we have:\\
\begin{align*}
 \mathfrak{G(x^{*},y^{*})}&=\mathfrak{G}(\Psi\mathfrak{x^{*}}, \Psi \mathfrak{y^{*}})\\
  &\le\xi_{1}\mathfrak{G(x^{*},y^{*})}+\xi_{2}[\mathfrak{G(x^{*}},\Psi \mathfrak{x^{*}})+\mathfrak{G(y^{*}},\Psi\mathfrak{y^{*})}].\\
  \end{align*}
   Since $\mathfrak{x^{*},y^{*}}$ are fixed points, we have $\mathfrak{G(x^{*}},\Psi \mathfrak{x^{*}})=\mathfrak{G(y^{*}},\Psi\mathfrak{y^{*})}=0.$\\
   \\
   Thus,  
   \begin{align*}
   \mathfrak{G(x^{*},y^{*})}\le\xi_{1} \mathfrak{G(x^{*},y^{*})}.\\
   \end{align*}
   Since $0\le\xi_{1}<1$, we have  $\mathfrak{G(x^{*},y^{*})}=0.$
   Thus, $\mathfrak{x^{*}=y^{*}}$ and hence $\mathfrak{x^{*}}$ is a unique fixed point of $\Psi$.
\end{proof}
\begin{Corollary}\label{22}
Let $(X,\mathfrak{G})$ be a complete $(\alpha,\beta)$- metric space. Let $\Psi:X\rightarrow X$ be contraction mapping that satisfies:
\begin{align*}
\mathfrak{G}(\Psi\mathfrak{x}, \Psi \mathfrak{y})\le K\mathfrak{G(x,y)}~\forall~\mathfrak{x,y}\in X,
\end{align*} where $0\le K<1$. Then $\Psi$ has a unique fixed point.
\end{Corollary}

\begin{proof}
The given contraction condition is a special case of the general $(\alpha,\beta)$-contraction condition with $\xi_{1}=K~and~\xi_{2}=0.$\\
Since $\xi_{1}+\xi_{2}=K<1$, the condition of the fixed point theorem are satisfied. Therefore by Theorem \ref{11}, $\Psi$ has a unique fixed point.
\end{proof}

\begin{Remark}
Corollary \ref{22}, shows that Banach fixed point theorem in classical metric spaces is a special case of the Theorem \ref{11}.
\end{Remark}

\begin{Definition}
Let $(X,\mathfrak{G})$ be an $(\alpha,\beta)$- metric space. Then a mapping $\Psi:X\rightarrow X$ is said to be a weak $(\alpha,\beta)$-contraction, if there exist constants $\xi_{1},\xi_{2}\ge 0$ with $\xi_{1}+\xi_{2}<1$, satisfies the following condition:
\begin{equation}\label{3}
\mathfrak{G}(\Psi\mathfrak{x}, \Psi \mathfrak{y})\le \xi_{1}\mathfrak{G(x,y)}+\xi_{2}max\Bigl[\mathfrak{G(x},\Psi \mathfrak{x}),\mathfrak{G(y,}\Psi\mathfrak{y)}\Bigr]~\forall~\mathfrak{x,y}\in X.
\end{equation}
\end{Definition}

\begin{Theorem}\label{33}
Let $(X,\mathfrak{G})$ be a complete $(\alpha,\beta)$- metric space. If $\Psi:X\rightarrow X$ is a weak $(\alpha,\beta)$-contraction. Then $\Psi$ has a unique fixed point.
\end{Theorem}

\begin{proof}
Let $\mathfrak{x_{\circ}}$ be an arbitrary point in $X$. Consider the sequence $\{\mathfrak{x_{n}}\}$ defined  by
\begin{equation}\label{4}
 \mathfrak{x_{n+1}}=\Psi\mathfrak{x_{n}}, \forall~n=0,1,2,3...
\end{equation}
Using (\ref{3}) and (\ref{4}), we get\\
\begin{align*}
\mathfrak{G(x_{n+1}, x_{n+2}})&=\mathfrak{G}(\Psi\mathfrak{x_{n}}, \Psi \mathfrak{x_{n+1}})\\
\\
&\le\xi_{1}\mathfrak{G(x_{n},x_{n+1})}+\xi_{2}max\Bigl[\mathfrak{G(x_{n}},\Psi \mathfrak{x_{n}}),\mathfrak{G(x_{n+1},}\Psi\mathfrak{x_{n+1})}\Bigr]\\
\\
&\le\xi_{1}\mathfrak{G(x_{n},x_{n+1})}+\xi_{2}max\Bigl[\mathfrak{G(x_{n}},\mathfrak{x_{n+1}}),\mathfrak{G(x_{n+1},} \mathfrak{x_{n+2})}\Bigr]\\
\\
&\le\xi_{1}\mathfrak{G(x_{n},x_{n+1})}+\xi_{2}\mathfrak{M_{1}}.
\end{align*}
Where, $\mathfrak{M_{1}}=max\Bigl[\mathfrak{G(x_{n}},\mathfrak{x_{n+1}}),\mathfrak{G(x_{n+1},} \mathfrak{x_{n+2})}\Bigr].$\\
\\
Now two case arises:\\
$\textbf{Case 1:}$ If  $\mathfrak{M_{1}}=\mathfrak{G(x_{n+1},} \mathfrak{x_{n+2})}$, then we have:\\
\begin{align*}
\mathfrak{G(x_{n+1}, x_{n+2}})&\le\xi_{1}\mathfrak{G(x_{n},x_{n+1})}+\xi_{2}\mathfrak{G(x_{n+1},} \mathfrak{x_{n+2})}\\
\\
(1-\xi_{2})\mathfrak{G(x_{n+1}, x_{n+2}})&\le\xi_{1}\mathfrak{G(x_{n},x_{n+1})}\\
\\
\mathfrak{G(x_{n+1}, x_{n+2}})&\le\frac{\xi_{1}}{(1-\xi_{2})}\mathfrak{G(x_{n},x_{n+1})}\\
\\
\mathfrak{G(x_{n+1}, x_{n+2}})&\le K\mathfrak{G(x_{n},x_{n+1})}.
\end{align*}

Where $K=\frac{\xi_{1}}{(1-\xi_{2})}<1$,
\begin{align*}
\mathfrak{G(x_{n}, x_{n+1}})&\le K\mathfrak{G(x_{n-1},x_{n})}\\
\\
&\le K^{2}\mathfrak{G(x_{n-2},x_{n-1})}\\
&\vdots\\
&\le K^{n}\mathfrak{G(x_{\circ},x_{1})}.
\end{align*}
$\textbf{Case 2:}$ If  $\mathfrak{M_{1}}=\mathfrak{G(x_{n}},\mathfrak{x_{n+1}})$, then we have:\\
\begin{align*}
\mathfrak{G(x_{n+1}, x_{n+2}})&\le\xi_{1}\mathfrak{G(x_{n},x_{n+1})}+\xi_{2}\mathfrak{G(x_{n}},\mathfrak{x_{n+1}})\\
\\
&\le(\xi_{1}+\xi_{2})\mathfrak{G(x_{n},x_{n+1})}\\
\\
&\le K\mathfrak{G(x_{n},x_{n+1})}.
\end{align*}
Where, $\xi_{1}+\xi_{2}=K<1,$
\begin{align*}
\mathfrak{G(x_{n}, x_{n+1}})&\le K\mathfrak{G(x_{n-1},x_{n})}.\\
\\
&\le K^{2}\mathfrak{G(x_{n-2},x_{n-1})}\\
&\vdots\\
&\le K^{n}\mathfrak{G(x_{\circ},x_{1})}.\\
\end{align*}
 Now, we apply the triangle inequality in $(\alpha,\beta)$- metric space to estimate $\mathfrak{G(x_{n}, x_{m}})$ for $m>n,$ $m,n\in\mathbb{N}.$\\
   \begin{align*}
    \mathfrak{G(x_{n}, x_{m}})&\le\alpha \mathfrak{G(x_{n}, x_{n+1}})+\beta \mathfrak{G(x_{n+1}, x_{m}})\\
     \\
    &\le\alpha \mathfrak{G(x_{n}, x_{n+1}})+\beta[\alpha \mathfrak{G(x_{n+1}, x_{n+2}})+\beta \mathfrak{G(x_{n+2}, x_{m}})]\\
    \\
     &=\alpha \mathfrak{G(x_{n}, x_{n+1}})+\alpha\beta\mathfrak{G(x_{n+1}, x_{n+2}})+\beta^{2} \mathfrak{G(x_{n+2}, x_{m}})\\
     \\
     &\le\alpha \mathfrak{G(x_{n}, x_{n+1}})+\alpha\beta\mathfrak{G(x_{n+1}, x_{n+2}})+\beta^{2}[\alpha \mathfrak{G(x_{n+2}, x_{n+3}})+\beta \mathfrak{G(x_{n+3}, x_{m}})]\\
      \\
      &\le\alpha \mathfrak{G(x_{n}, x_{n+1}})+\alpha\beta\mathfrak{G(x_{n+1}, x_{n+2}})+\alpha\beta^{2} \mathfrak{G(x_{n+2}, x_{n+3}})+\cdots\\
      \\
      &\le\alpha K^{n} \mathfrak{G(x_{\circ},x_{1})}+\alpha\beta K^{n+1} \mathfrak{G(x_{\circ},x_{1})}+\alpha\beta^{2} K^{n+2} \mathfrak{G(x_{\circ},x_{1})}+\cdots\\
      \\
      &=\alpha K^{n} \mathfrak{G(x_{\circ},x_{1})}\Bigl[1+\beta K+(\beta K)^{2}+\cdots\Bigr]\\
      \\
      &= \frac {\alpha K^{n}}{1-\beta K} \mathfrak{G(x_{\circ},x_{1}}).
      \end{align*}
      Therefore, we have\\
      \begin{align*}
     \mathfrak{G(x_{n}, x_{m}}) \le \frac {\alpha K^{n}}{1-\beta K} \mathfrak{G(x_{\circ},x_{1}}).
      \end{align*}
      Taking $\lim_{n,m\to\infty}$, Since $K<1,\lim_{n\to\infty}\frac{\alpha K^{n}}{1-\beta K} \mathfrak{G(x_{\circ},x_{1})}=0.$\\
Thus, $\lim_{n,m\to\infty}\mathfrak{G(x_{n}, x_{m}})=0.$\\
This shows that $\{\mathfrak{x_{n}}\}$ is a cauchy sequence in $X$. Since $X$ is Complete $(\alpha, \beta)$-metric space, so there exist $\mathfrak{x^{*}}\in X$ such that $\mathfrak{x_{n}}\rightarrow\mathfrak{x^{*}}.$\\
\\
To show that $\Psi\mathfrak{x^{*}}=\mathfrak{x^{*}}$, Consider
\begin{align*}
\mathfrak{G(x^{*}},\Psi\mathfrak{x^{*}})&\le\alpha \mathfrak{G(x^{*},x_{n+1}})+\beta\mathfrak{G(x_{n+1}},\Psi \mathfrak{x^{*}})\\
\\
&\le\alpha \mathfrak{G(x^{*},x_{n+1}})+\beta\mathfrak{G(}\Psi\mathfrak{ x_{n}},\Psi\mathfrak{x^{*}})\\
\\
&\le\alpha \mathfrak{G(x^{*},x_{n+1}})+\beta\Bigl[\xi_{1}\mathfrak{G(x_{n},x^{*}})+\xi_{2}max[\mathfrak{G(x_{n}},\Psi \mathfrak{x_{n}}),\mathfrak{G(x^{*}},\Psi\mathfrak{x^{*})}]\Bigr]\\
\\
&=\alpha \mathfrak{G(x^{*},x_{n+1}})+\beta\xi_{1}\mathfrak{G(x_{n},x^{*}})+\beta\xi_{2}max[\mathfrak{G(x_{n}},\Psi \mathfrak{x_{n}})+\mathfrak{G(x^{*}},
\Psi\mathfrak{x^{*})}]\\
\\
&=\alpha \mathfrak{G(x^{*},x_{n+1}})+\beta\xi_{1}\mathfrak{G(x_{n},x^{*}})+\beta\xi_{2}\mathfrak{M_{2}}.\\
\end{align*}
Where, $\mathfrak{M_{2}}=max[\mathfrak{G(x_{n}},\Psi \mathfrak{x_{n}}),\mathfrak{G(x^{*}}, \Psi\mathfrak{x^{*})}].$\\
\\
Now, two case arises:\\
\\
$\textbf{Case 1:}$ When $\mathfrak{M_{2}}=\mathfrak{G(x_{n}},\Psi \mathfrak{x_{n}})$, then\\
\begin{align*}
\mathfrak{G(x^{*}},\Psi\mathfrak{x^{*}})&\le\alpha \mathfrak{G(x^{*},x_{n+1}})+\beta\xi_{1}\mathfrak{G(x_{n},x^{*}})+\beta\xi_{2}\mathfrak{G(x_{n}},\Psi \mathfrak{x_{n}})\\
\\
&\le\alpha \mathfrak{G(x^{*},x_{n+1}})+\beta\xi_{1}\mathfrak{G(x_{n},x^{*}})+\beta\xi_{2}[\alpha\mathfrak{G(x_{n},x^{*})}+\beta\mathfrak{G(x^{*}},\Psi x_{n})]\\
\\
&=\alpha \mathfrak{G(x^{*},x_{n+1}})+\beta\xi_{1}\mathfrak{G(x_{n},x^{*}})+\alpha\beta\xi_{2}\mathfrak{G(x_{n},x^{*})}+\beta^{2}\xi_{2}\mathfrak{G(x^{*}},x_{n+1})\\
\\
&=(\alpha+\beta^{2}\xi_{2}) \mathfrak{G(x^{*},x_{n+1}})+\beta(\xi_{1}+\xi_{2}\alpha)\mathfrak{G(x_{n},x^{*})}.
\end{align*}\\
Therefore, we have
\begin{align*}
\mathfrak{G(x^{*}},\Psi\mathfrak{x^{*}})\le(\alpha+\beta^{2}\xi_{2}) \mathfrak{G(x^{*},x_{n+1}})+\beta(\xi_{1}+\xi_{2}\alpha)\mathfrak{G(x_{n},x^{*})}.
\end{align*}\\
Taking $\lim n\to\infty,$ we get\\
$\lim_{n\to\infty}\mathfrak{G(x^{*}},\Psi\mathfrak{x^{*}})=0.$\\
$\Rightarrow \mathfrak{x^{*}}=\Psi\mathfrak{x^{*}}.$\\
Hence,  $\Psi\mathfrak{x^{*}}=\mathfrak{x^{*}}.$\\
\\
$\textbf{Case 2:}$ When $\mathfrak{M_{2}}=\mathfrak{G(x^{*}},\Psi\mathfrak{x^{*})}$, then\\
\begin{align*}
\mathfrak{G(x^{*}},\Psi\mathfrak{x^{*}})&\le\alpha \mathfrak{G(x^{*},x_{n+1}})+\beta\xi_{1}\mathfrak{G(x_{n},x^{*}})+\beta\xi_{2}\mathfrak{G(x^{*}},
\Psi\mathfrak{x^{*})}\\
\\
(1-\beta\xi_{2})\mathfrak{G(x^{*}},\Psi\mathfrak{x^{*}})&\le\alpha \mathfrak{G(x^{*},x_{n+1}})+\beta\xi_{1}\mathfrak{G(x_{n},x^{*}}).\\
\end{align*}\\
Therefore, we have
\begin{align*}
\mathfrak{G(x^{*}},\Psi\mathfrak{x^{*}})\le\frac{\alpha}{(1-\beta\xi_{2})} \mathfrak{G(x^{*},x_{n+1}})+\frac{\beta\xi_{1}}{(1-\beta\xi_{2})}\mathfrak{G(x_{n},x^{*}}).\\
\end{align*}\\
Taking $\lim n\to\infty,$ we get\\
$\lim_{n\to\infty}\mathfrak{G(x^{*}},\Psi\mathfrak{x^{*}})=0.$\\
$\Rightarrow \mathfrak{x^{*}}=\Psi\mathfrak{x^{*}}.$\\
Hence,  $\Psi\mathfrak{x^{*}}=\mathfrak{x^{*}}.$\\
Therefore, $\mathfrak{x^{*}}$ is a fixed point of $\Psi.$\\
\\
$\textbf{Uniqueness of Fixed Point:}$\\
\\
Suppose  $\mathfrak{y^{*}}\in X$ such that $\Psi(\mathfrak{y^{*})=y^{*}}$. Then, we have:\\
\begin{align*}
\mathfrak{G(x^{*},y^{*})}&=\mathfrak{G}(\Psi\mathfrak{x^{*}}, \Psi \mathfrak{y^{*}})\\
\\
  &\le\xi_{1}\mathfrak{G(x^{*},y^{*})}+\xi_{2}max\Bigl[\mathfrak{G(x^{*}},\Psi \mathfrak{x^{*}}),\mathfrak{G(y^{*}},\Psi\mathfrak{y^{*})}\Bigr]\\
  \\
  &\le\xi_{1}\mathfrak{G(x^{*},y^{*})}+\xi_{2}max\Bigl[\mathfrak{G(x^{*},x^{*}}),\mathfrak{G(y^{*},y^{*})}\Bigr].\\
  \end{align*}
  Therefore, we get
  \begin{align*}
\mathfrak{G(x^{*},y^{*})}\le\xi_{1}\mathfrak{G(x^{*},y^{*})}.
\end{align*}
Since $0\le\xi_{1}<1$, we have  $\mathfrak{G(x^{*},y^{*})}=0.$
 Thus, $\mathfrak{x^{*}=y^{*}}$ and hence $\mathfrak{x^{*}}$ is a unique fixed point of $\Psi$.
 \end{proof}
 
 \begin{Example}
 Let $X=[0,1]$ with a mapping $\Psi:X\rightarrow X$ defined by $\Psi\mathfrak{(x)}=\frac{\mathfrak{x}}{4}$.\\
 Define $\mathfrak{G(x,y)=|x-y|}~\forall~\mathfrak{x,y}\in X.$\\
 Then $\mathfrak{G(x,y)=|x-y|}$ define a valid  $(\alpha,\beta)$-metric for $\alpha=2, \beta=1.$\\
 Now, we show that $\Psi$ satisfies the weak $(\alpha,\beta)$-contraction condition:\\
 \begin{equation}\label{5}
 i.e; \mathfrak{G}(\Psi\mathfrak{x}, \Psi \mathfrak{y})\le \xi_{1}\mathfrak{G(x,y)}+\xi_{2}max\Bigl[\mathfrak{G(x},\Psi \mathfrak{x}),\mathfrak{G(y,}\Psi\mathfrak{y)}\Bigr]~\forall~\mathfrak{x,y}\in X.
 \end{equation}
 Take $\xi_{1}=\frac{1}{4}$ and $\xi_{2}=0$, then\\
 \begin{align*}
 \Bigl|\frac{\mathfrak{x}}{4}-\frac{\mathfrak{y}}{4}\Bigr|\le\frac{1}{4}\Bigl|\mathfrak{x-y}\Bigr|+0\times max\Biggl[\Bigl|\mathfrak{x}-\frac{\mathfrak{x}}{4}\Bigr|,\Bigl|\mathfrak{y}-\frac{\mathfrak{y}}{4}\Bigr|\Biggr].
 \end{align*}
 Which implies 
 \begin{align*}
 \Bigl|\frac{\mathfrak{x-y}}{4}\Bigr|\le \Bigl|\frac{\mathfrak{x-y}}{4}\Bigr|.
 \end{align*}
 Thus, $\Psi$ satisfy (\ref{5}) with  $\xi_{1}=\frac{1}{4},\xi_{2}=0$ and $\xi_{1}+\xi_{2}=\frac{1}{4}<1.$\\
 Therefore by Theorem \ref{33}, $\Psi$ has a fixed point.
 \begin{align*}
i.e.,  \Psi\mathfrak{(x)}&=\mathfrak{x}\\
\Rightarrow\frac{\mathfrak{x}}{4}&=\mathfrak{x}\\
\Rightarrow\frac{\mathfrak{x}}{4}-\mathfrak{x}&=0\\
\Rightarrow\mathfrak{x}&=0.
 \end{align*}
 Therefore $\mathfrak{x^{*}}=0$ is a unique fixed point of $\Psi.$\\
 Now, we see how the sequence $\{\mathfrak{x_{n}}\}$ defined by $\mathfrak{x_{n}}=\Psi\mathfrak{x_{n-1}}$ converges to $0$.\\
 Let $\mathfrak{x_{\circ}}=1$.
 Then,
 \begin{align*}
  \mathfrak{x_{1}}&=\Psi\mathfrak{(x_{\circ})}=\frac{1}{4}=0.25\\
  \mathfrak{x_{2}}&=\Psi\mathfrak{(x_{1})}=\frac{0.25}{4}=0.0625\\
  \mathfrak{x_{3}}&=\Psi\mathfrak{(x_{2})}=\frac{0.0625}{4}=0.015625\\
  \vdots
 \end{align*}
 Thus,  $\{\mathfrak{x_{n}}\}$ is a monotone decreasing sequence converges to 0 as $n\rightarrow\infty.$
 \end{Example}
 
\begin{Theorem}
$\textbf{(Reich's fixed point theorem in $(\alpha,\beta)$-metric spaces)}$\\
Let $(X,\mathfrak{G})$ be a complete $(\alpha,\beta)$- metric space. Let $\Psi:X\rightarrow X$ be a self map that satisfies the following contractive condition, i.e;  there exist constants $\xi_{1},\xi_{2},\xi_{3}\ge 0$ with $\xi_{1}+\xi_{2}+\xi_{3}<1,$ such that:
\begin{equation}\label{6}
\mathfrak{G}(\Psi\mathfrak{x}, \Psi \mathfrak{y})\le \xi_{1}\mathfrak{G(x,y)}+\xi_{2}\mathfrak{G(x},\Psi \mathfrak{x})+\xi_{3}\mathfrak{G(y,}\Psi\mathfrak{y)}~\forall~\mathfrak{x,y}\in X.
\end{equation}
 Then $\Psi$ has a unique fixed point.
\end{Theorem}
\begin{proof}
Let $\mathfrak{x_{\circ}}$ be any arbitrary point in $X$. Define a sequence $\{\mathfrak{x_{n}}\}$ by
\begin{equation}\label{7}
\mathfrak{x_{n}}=\Psi\mathfrak{x_{n-1}}~\forall~n\ge1.
\end{equation}
By (\ref{6}) and (\ref{7}), we obtain that
\begin{align*}
\mathfrak{G(x_{n}, x_{n+1}})&=\mathfrak{G}(\Psi\mathfrak{x_{n-1}}, \Psi \mathfrak{x_{n}})\\
\\
 &\le\xi_{1}\mathfrak{G(x_{n-1},x_{n})}+\xi_{2}\mathfrak{G(x_{n-1}},\Psi \mathfrak{x_{n-1}})+\xi_{3}\mathfrak{G(x_{n},}\Psi\mathfrak{x_{n})}\\
 \\
 &=\xi_{1}\mathfrak{G(x_{n-1},x_{n})}+\xi_{2}\mathfrak{G(x_{n-1}},\mathfrak{x_{n}})+\xi_{3}\mathfrak{G(x_{n},}\mathfrak{x_{n+1})}\\
 \\
 (1-\xi_{3})\mathfrak{G(x_{n}, x_{n+1}})&\le(\xi_{1}+\xi_{2})\mathfrak{G(x_{n-1}},\mathfrak{x_{n}}).\\
\end{align*}
Since $\xi_{1}+\xi_{2}+\xi_{3}<1$, we have  $(1-\xi_{3})\ge(\xi_{1}+\xi_{2})$. Thus\\
\begin{align*}
\mathfrak{G(x_{n}, x_{n+1}})&\le\frac{(\xi_{1}+\xi_{2})}{(1-\xi_{3})}\mathfrak{G(x_{n-1}},\mathfrak{x_{n}})\\
\\
\mathfrak{G(x_{n}, x_{n+1}})&\le K \mathfrak{G(x_{n-1}},\mathfrak{x_{n}}).
\end{align*}
Put $K=\frac{(\xi_{1}+\xi_{2})}{(1-\xi_{3})}<1$, we have
\begin{align*}
\mathfrak{G(x_{n}, x_{n+1}})&\le K \mathfrak{G(x_{n-1}},\mathfrak{x_{n}})\\
&\le K^{2} \mathfrak{G(x_{n-2}},\mathfrak{x_{n-1}})\\
&\vdots\\
&\le K^{n} \mathfrak{G(x_{\circ}},\mathfrak{x_{1}}).\\
\end{align*}
\\
Consider $\mathfrak{m,n}\in\mathbb{N}$ with $\mathfrak{m>n}$, we have
 \begin{align*}
    \mathfrak{G(x_{n}, x_{m}})&\le\alpha \mathfrak{G(x_{n}, x_{n+1}})+\beta \mathfrak{G(x_{n+1}, x_{m}})\\
     \\
    &\le\alpha \mathfrak{G(x_{n}, x_{n+1}})+\beta[\alpha \mathfrak{G(x_{n+1}, x_{n+2}})+\beta \mathfrak{G(x_{n+2}, x_{m}})]\\
    \\
     &\le\alpha \mathfrak{G(x_{n}, x_{n+1}})+\alpha\beta\mathfrak{G(x_{n+1}, x_{n+2}})+\beta^{2}[\alpha \mathfrak{G(x_{n+2}, x_{n+3}})+\beta \mathfrak{G(x_{n+3}, x_{m}})]\\
      \\
      &\le\alpha \mathfrak{G(x_{n}, x_{n+1}})+\alpha\beta\mathfrak{G(x_{n+1}, x_{n+2}})+\alpha\beta^{2} \mathfrak{G(x_{n+2}, x_{n+3}})+\cdots\\
      \\
      &\le\alpha K^{n} \mathfrak{G(x_{\circ},x_{1})}+\alpha\beta K^{n+1} \mathfrak{G(x_{\circ},x_{1})}+\alpha\beta^{2} K^{n+2} \mathfrak{G(x_{\circ},x_{1})}+\cdots\\
      \\
      &=\alpha K^{n} \mathfrak{G(x_{\circ},x_{1})}\Bigl[1+\beta K+(\beta K)^{2}+\cdots\Bigr]\\
      \\
      &= \frac {\alpha K^{n}}{1-\beta K} \mathfrak{G(x_{\circ},x_{1}}).
      \end{align*}
      Therefore, we have
      \begin{align*}
     \mathfrak{G(x_{n}, x_{m}}) \le \frac {\alpha K^{n}}{1-\beta K} \mathfrak{G(x_{\circ},x_{1}}).
      \end{align*}
      Since $K<1,\lim_{n\to\infty}\frac{\alpha K^{n}}{1-\beta K} \mathfrak{G(x_{\circ},x_{1})}=0.$\\
      \\
Thus, $\lim_{n,m\to\infty}\mathfrak{G(x_{n}, x_{m}})=0.$\\
This shows that $\{\mathfrak{x_{n}}\}$ is a cauchy sequence in $X$. Since $X$ is Complete $(\alpha, \beta)$-metric space, so there exist $\mathfrak{x^{*}}\in X$ such that $\mathfrak{x_{n}}\rightarrow\mathfrak{x^{*}}.$\\
\\
To show that $\Psi\mathfrak{x^{*}}=\mathfrak{x^{*}}$, Consider
\begin{align*}
\mathfrak{G(x^{*}},\Psi\mathfrak{x^{*}})&\le\alpha \mathfrak{G(x^{*},x_{n+1}})+\beta\mathfrak{G(x_{n+1}},\Psi \mathfrak{x^{*}})\\
\\
&\le\alpha \mathfrak{G(x^{*},x_{n+1}})+\beta\mathfrak{G(}\Psi\mathfrak{ x_{n}},\Psi\mathfrak{x^{*}})\\
\\
&\le\alpha \mathfrak{G(x^{*},x_{n+1}})+\beta\Bigl[\xi_{1}\mathfrak{G(x_{n},x^{*})}+\xi_{2}\mathfrak{G(x_{n}},\Psi \mathfrak{x_{n}})+\xi_{3}\mathfrak{G(x^{*},}\Psi\mathfrak{x^{*})}\Bigr]\\
\\
&=\alpha \mathfrak{G(x^{*},x_{n+1}})+\beta\xi_{1}\mathfrak{G(x_{n},x^{*})}+\beta\xi_{2}\mathfrak{G(x_{n}},\mathfrak{x_{n+1}})+\beta\xi_{3}\mathfrak{G(x^{*},}\Psi\mathfrak{x^{*})}\\
\\
(1-\beta\xi_{3})\mathfrak{G(x^{*},}\Psi\mathfrak{x^{*})}&\le\alpha \mathfrak{G(x^{*},x_{n+1}})+\beta\xi_{1}\mathfrak{G(x_{n},x^{*})}+\beta\xi_{2}\Bigl[\alpha\mathfrak{G(x_{n}},\mathfrak{x^{*}})+\beta\mathfrak{G(x^{*}},\mathfrak{x_{n+1}})\Bigr]\\
\\
&=\alpha \mathfrak{G(x^{*},x_{n+1}})+\beta\xi_{1}\mathfrak{G(x_{n},x^{*})}+\alpha\beta\xi_{2}\mathfrak{G(x_{n}},\mathfrak{x^{*}})+\beta^{2}\xi_{2}\mathfrak{G(x^{*}},\mathfrak{x_{n+1}}).\\
\end{align*}
Which implies,
\begin{align*}
\mathfrak{G(x^{*}},\Psi\mathfrak{x^{*}})\le\frac{(\alpha+\beta^{2}\xi_{2})}{(1-\beta\xi_{3})}\mathfrak{G(x^{*},x_{n+1}})+\frac{\beta(\xi_{1}+\alpha\xi_{2})}{(1-\beta\xi_{3})}\mathfrak{G(x_{n}},\mathfrak{x^{*}}).
\end{align*}
Taking $\lim_{n\to\infty}$, we get\\
\\
$\lim_{n\to\infty}\mathfrak{G(x^{*}},\Psi\mathfrak{x^{*}})=0.$\\
$\Rightarrow \mathfrak{x^{*}}=\Psi\mathfrak{x^{*}}.$\\
Therefore, $\mathfrak{x^{*}}$ is a limit point of $\Psi$.\\

$\textbf{Uniqueness of Fixed Point:}$\\
Suppose  $\mathfrak{y^{*}}$ be another fixed point of $\Psi$, then we have
\begin{align*}
 \Psi(\mathfrak{y^{*})=y^{*}} and~\mathfrak{G(x^{*},y^{*})}&=\mathfrak{G}(\Psi\mathfrak{x^{*}}, \Psi \mathfrak{y^{*}})\\
 \\
  &\le\xi_{1}\mathfrak{G(x^{*},y^{*})}+\xi_{2}\mathfrak{G(x^{*}},\Psi \mathfrak{x^{*}})+\xi_{3}\mathfrak{G(y^{*}},\Psi\mathfrak{y^{*})}\\
  \\
  &=\xi_{1}\mathfrak{G(x^{*},y^{*})}+\xi_{2}\mathfrak{G(x^{*}},\mathfrak{x^{*}})+\xi_{3}\mathfrak{G(y^{*}},\mathfrak{y^{*})}.
\end{align*}
 Which implies,
  \begin{align*}
\mathfrak{G(x^{*},y^{*})}\le\xi_{1}\mathfrak{G(x^{*},y^{*})}.
\end{align*}
Since $\xi_{1}+\xi_{2}+\xi_{3}<1$, we have  $\xi_{1}<1.$\\
Therefore, 
\begin{align*}
(1-\xi_{1})\mathfrak{G(x^{*},y^{*})}\le 0.
\end{align*}
Since $(1-\xi_{1})>0$, which follows that $\mathfrak{G(x^{*},y^{*})}= 0$. Thus, $\mathfrak{x^{*}=y^{*}}$ and hence $\mathfrak{x^{*}}$ is a unique fixed point of $\Psi$.
\end{proof}

\begin{Theorem}\label{111}
$\textbf{(Kannan's fixed point theorem in $(\alpha,\beta)$-metric spaces)}$\\
Let $(X,\mathfrak{G})$ be a complete $(\alpha,\beta)$- metric space. Let $\Psi:X\rightarrow X$ be a self map that satisfies the following contractive condition, i.e;  there exist constant $\lambda\in [0,\frac{1}{2})$ such that:
\begin{equation}\label{8}
\mathfrak{G}(\Psi\mathfrak{x}, \Psi \mathfrak{y})\le\lambda\Bigl[\mathfrak{G(x},\Psi \mathfrak{x})+\mathfrak{G(y,}\Psi\mathfrak{y)}\Bigr]~\forall~\mathfrak{x,y}\in X.
\end{equation}
 Then $\Psi$ has a unique fixed point.
\end{Theorem}
\begin{proof}
Let $\mathfrak{x_{\circ}}$ be any arbitrary point in $X$. Consider the sequence $\{\mathfrak{x_{n}}\}$ defined by
\begin{equation}\label{9}
\mathfrak{x_{n}}=\Psi\mathfrak{x_{n-1}}~\forall~n=1,2,3\dots
\end{equation}
\\
Using the inequality (\ref{8}) and (\ref{9}), we have
\begin{align*}
\mathfrak{G(x_{n}, x_{n+1}})&=\mathfrak{G}(\Psi\mathfrak{x_{n-1}}, \Psi \mathfrak{x_{n}})\\
\\
& \le\lambda\Bigl[\mathfrak{G(x_{n-1}},\Psi\mathfrak{x_{n-1}})+\mathfrak{G(x_{n},}\Psi\mathfrak{x_{n})}\Bigr]\\
\\
&= \lambda\mathfrak{G(x_{n-1}},\mathfrak{x_{n}})+\lambda\mathfrak{G(x_{n}}, \mathfrak{x_{n+1})}\\
\\
(1-\lambda)\mathfrak{G(x_{n}, x_{n+1}})&\le\lambda\mathfrak{G(x_{n-1}},\mathfrak{x_{n}})\\
\\
\mathfrak{G(x_{n}, x_{n+1}})&\le\frac{\lambda}{(1-\lambda)}\mathfrak{G(x_{n-1}},\mathfrak{x_{n}}).\\
\end{align*}
Let $\frac{\lambda}{(1-\lambda)}=K<1$. Therefore, we have
\begin{align*}
\mathfrak{G(x_{n}, x_{n+1}})&\le K \mathfrak{G(x_{n-1}},\mathfrak{x_{n}})\\
\\
&\le K^{2} \mathfrak{G(x_{n-2}},\mathfrak{x_{n-1}})\\
&\vdots\\
&\le K^{n} \mathfrak{G(x_{\circ}},\mathfrak{x_{1}}).\\
\end{align*}
Consider $\mathfrak{m,n}\in\mathbb{N}$ with $\mathfrak{m>n}$, we have\\
 \begin{align*}
    \mathfrak{G(x_{n}, x_{m}})&\le\alpha \mathfrak{G(x_{n}, x_{n+1}})+\beta \mathfrak{G(x_{n+1}, x_{m}})\\
     \\
    &\le\alpha \mathfrak{G(x_{n}, x_{n+1}})+\beta[\alpha \mathfrak{G(x_{n+1}, x_{n+2}})+\beta \mathfrak{G(x_{n+2}, x_{m}})]\\
    \\
     &\le\alpha \mathfrak{G(x_{n}, x_{n+1}})+\alpha\beta\mathfrak{G(x_{n+1}, x_{n+2}})+\beta^{2}[\alpha \mathfrak{G(x_{n+2}, x_{n+3}})+\beta \mathfrak{G(x_{n+3}, x_{m}})]\\
      \\
      &\le\alpha \mathfrak{G(x_{n}, x_{n+1}})+\alpha\beta\mathfrak{G(x_{n+1}, x_{n+2}})+\alpha\beta^{2} \mathfrak{G(x_{n+2}, x_{n+3}})+\cdots\\
      \\
      &\le\alpha K^{n} \mathfrak{G(x_{\circ},x_{1})}+\alpha\beta K^{n+1} \mathfrak{G(x_{\circ},x_{1})}+\alpha\beta^{2} K^{n+2} \mathfrak{G(x_{\circ},x_{1})}+\cdots\\
      \\
      &=\alpha K^{n} \mathfrak{G(x_{\circ},x_{1})}\Bigl[1+\beta K+(\beta K)^{2}+\cdots\Bigr]\\
      \\
      &= \frac {\alpha K^{n}}{1-\beta K} \mathfrak{G(x_{\circ},x_{1}}).
      \end{align*}
      Therefore, we have\\
      \begin{align*}
     \mathfrak{G(x_{n}, x_{m}}) \le \frac {\alpha K^{n}}{1-\beta K} \mathfrak{G(x_{\circ},x_{1}}).
      \end{align*}
      Since $K<1,\lim_{n\to\infty}\frac{\alpha K^{n}}{1-\beta K} \mathfrak{G(x_{\circ},x_{1})}=0.$\\
      \\
Thus, $\lim_{n,m\to\infty}\mathfrak{G(x_{n}, x_{m}})=0.$\\
This shows that $\{\mathfrak{x_{n}}\}$ is a cauchy sequence in $X$. Since $X$ is Complete $(\alpha, \beta)$-metric space, so there exist $\mathfrak{x^{*}}\in X$ such that $\mathfrak{x_{n}}\rightarrow\mathfrak{x^{*}}.$\\
\\
To show that $\Psi\mathfrak{x^{*}}=\mathfrak{x^{*}}$, Consider
\begin{align*}
\mathfrak{G(x^{*}},\Psi\mathfrak{x^{*}})&\le\alpha \mathfrak{G(x^{*},x_{n+1}})+\beta\mathfrak{G(x_{n+1}},\Psi \mathfrak{x^{*}})\\
\\
&\le\alpha \mathfrak{G(x^{*},x_{n+1}})+\beta\mathfrak{G(}\Psi\mathfrak{ x_{n}},\Psi\mathfrak{x^{*}})\\
\\
&\le\alpha \mathfrak{G(x^{*},x_{n+1}})+\beta\Bigl[\lambda[\mathfrak{G(x_{n}},\Psi \mathfrak{x_{n}})+\mathfrak{G(x^{*}},\Psi\mathfrak{x^{*})}]\Bigr]\\
\\
&=\alpha \mathfrak{G(x^{*},x_{n+1}})+\lambda\beta\mathfrak{G(x_{n}},\mathfrak{x_{n+1}})+\lambda\beta\mathfrak{G(x^{*}},\Psi\mathfrak{x^{*})}\\
\\
(1-\lambda\beta)\mathfrak{G(x^{*}},\Psi\mathfrak{x^{*}})&\le\alpha \mathfrak{G(x^{*},x_{n+1}})+\lambda\beta\Bigl[\alpha \mathfrak{G(x_{n},x^{*}})+\beta\mathfrak{G(x^{*}},\mathfrak{x_{n+1}})\Bigr]\\
\\
&=\alpha \mathfrak{G(x^{*},x_{n+1}})+\lambda\alpha\beta \mathfrak{G(x_{n},x^{*}})+\lambda\beta^{2}\mathfrak{G(x^{*}},\mathfrak{x_{n+1}}).
\end{align*}
Therefore,
\begin{align*}
\mathfrak{G(x^{*}},\Psi\mathfrak{x^{*}})\le\frac{(\alpha+\beta^{2}\lambda)}{(1-\lambda\beta)} \mathfrak{G(x^{*},x_{n+1}})+\frac{\alpha\beta\lambda}{(1-\lambda\beta)} \mathfrak{G(x_{n},x^{*}}).
\end{align*}
Taking $\lim_{n\to\infty}$, we get\\
\\
$\lim_{n\to\infty}\mathfrak{G(x^{*}},\Psi\mathfrak{x^{*}})=0.$\\
$\Rightarrow \mathfrak{x^{*}}=\Psi\mathfrak{x^{*}}.$\\
Therefore, $\mathfrak{x^{*}}$ is a limit point of $\Psi$.\\

$\textbf{Uniqueness of Fixed Point:}$\\
Suppose  $\mathfrak{y^{*}}$ be another fixed point of $\Psi$, then we have
\begin{align*}
 \Psi(\mathfrak{y^{*})=y^{*}} and~\mathfrak{G(x^{*},y^{*})}&=\mathfrak{G}(\Psi\mathfrak{x^{*}}, \Psi \mathfrak{y^{*}})\\
 \\
 &\le K\Bigl[\mathfrak{G(x^{*}},\Psi \mathfrak{x^{*}})+\mathfrak{G(y^{*}},\Psi\mathfrak{y^{*})}\Bigr]\\
 \\
 &\le K \mathfrak{G(x^{*}}, \mathfrak{x^{*}})+K\mathfrak{G(y^{*}},\mathfrak{y^{*})}.\\
\end{align*}
 Since $\mathfrak{x^{*}}$ and $\mathfrak{y^{*}}$ are fixed points. Therefore we have\\
  \begin{align*}
\mathfrak{G(x^{*},y^{*})}\le0.
\end{align*}
Thus, $\mathfrak{x^{*}=y^{*}}$ and hence $\mathfrak{x^{*}}$ is a unique fixed point of $\Psi$.
\end{proof}

\begin{Remark}
If we take $\xi_{1}=0$ and $\xi_{2}=\lambda$ in Theorem \ref{11}, then~$Theorem~\ref{11}\Rightarrow Theorem~\ref{111}.$Therefore, we can conclude that a Kannan-type contraction is a special case of the $(\alpha,\beta)$-contraction.
\end{Remark}
\section{Applications to Fredholm integral and non-linear differential equations}\label{sec4}
In this section, as applications of Corollary \ref{22}, we discuss the existence of solution for the Fredholm integral and non-linear differential equations:
\subsection{Application to Fredholm integral equation:}
Integral equations have many applications in almost every field of science and engineering to explain phenomena such as signal processing, population dynamics, heat transfer, and control systems. Many initial and boundary-value problems can be simply transformed into integral equations. Because the subject has numerous applications, it has piqued the interest of many researchers in the past where they solved nonlinear Fredholm integral equations numerically by utilizing various approximate methods (see \cite{1010,1011}). Alturk \cite{1012} presented a simple and effective method for obtaining solutions for a rather wide class of Fredholm integral equations of the second kind. In 2016, Li and Huang \cite{1013} developed a novel approach to solve nonlinear Fredholm integral equations. Recently, Tiwari and Patel \cite{1014}, studied the existence and uniqueness of a solution to the nonlinear Fredholm integral equations.\\
\\ 
Let $\mathcal{X}=C([m,n],\mathbb{R}^{+})$ be the space of all continuous real valued function defined on $[m,n]$ with the $(\alpha,\beta)$-metric defined by
\begin{align*}
\mathfrak{G(u,v)}=Sup_{\mathfrak{t}\in[m,n]} \big|\mathfrak{u(t)-v(t)}\big|~\forall ~\mathfrak{u,v}\in\mathcal{X}.
\end{align*}
It is evident that $(\mathcal{X},\mathfrak{G})$ is a complete $(\alpha,\beta)$-metric space.\\
\\
Consider the Fredhlom Integral Equation
\begin{equation}\label{10}
\mathfrak{u(s)}=\int_{m}^{n}\mathfrak{F(s,t,u(t))}d\mathfrak{t},
\end{equation}
where $m,n\in\mathbb{R},\mathfrak{u}\in\mathcal{X}~and~\mathfrak{F}:[m,n]^{2}\times \mathbb{R}\rightarrow\mathbb{R}^{+}$ is continuous.\\
\\
Now, Define a mapping $\Psi:\mathcal{X\rightarrow X}$ by
\begin{align*}
\Psi\mathfrak{u(s)}=\int_{m}^{n}\mathfrak{F(s,t,u(t))}d\mathfrak{t},~\mathfrak{s,t}\in[m,n].
\end{align*}
Suppose that the kernal $\mathfrak{F}$ satisfies
\begin{align*}
\big|\mathfrak{F(s,t,u(t))}-\mathfrak{F(s,t,v(t))}\big|\le \Lambda\big|\mathfrak{u(t)-v(t)}\big|,
\end{align*}
for some constant $\Lambda>0~and~\forall~\mathfrak{s,t}\in[m,n],~\mathfrak{u,v}\in\mathbb{R}.$\\
Now for $\mathfrak{u,v}\in\mathcal{X}$ and $\mathfrak{s,t}\in[m,n]$, we have
\begin{align*}
\big|\Psi\mathfrak{u(s)}-\Psi\mathfrak{v(s)}\big|&=\bigg|\int_{m}^{n}\mathfrak{F(s,t,u(t))}d\mathfrak{t}-\int_{m}^{n}\mathfrak{F(s,t,v(t))}d\mathfrak{t}\bigg|\\
\\
&=\int_{m}^{n}\big|(\mathfrak{F(s,t,u(t))}-\mathfrak{F(s,t,v(t))})\big|d\mathfrak{t}\\
\\
&\le\Lambda\int_{m}^{n}\big|\mathfrak{u(t)-v(t)}\big|d\mathfrak{t}.
\end{align*}
Taking Supremum over $\mathfrak{t}\in[m,n]$, we get
\begin{align*}
Sup_{\mathfrak{t}\in[m,n]}\big|\Psi\mathfrak{u(s)}-\Psi\mathfrak{v(s)}\big|&\le Sup_{\mathfrak{t}\in[m,n]}\Lambda\int_{m}^{n}\big|\mathfrak{u(t)-v(t)}\big|d\mathfrak{t}\\
\\
\mathfrak{G}(\Psi\mathfrak{u},\Psi\mathfrak{v})&\le\Lambda(n-m)\mathfrak{G(u,v)}.
\end{align*}
Hence, all the hypothesis of Corollary \ref{22}, are satisfied. Therefore $\Psi$ has a unique fixed point, which is the solution of equation \ref{10}.
\subsection{Application to non-linear differential equation:}
Nonlinear differential equations play a fundamental role in science and engineering, as they describe complex phenomena that linear equations cannot capture. These equations have received a lot of interest nowadays for use in a variety of fields. Specifically, cantilever structures are studied using a mathematical model based on a nonlinear differential equation \cite{1015,1016} with shifting singular points. In \cite{1017}, a third-order differential equation is examined in an implicit form in order to investigate wave processes in elastic beams. In 2019, Imad et al. \cite{1018} examine the existence of solution of a second order ordinary differential equation. In this work, we have find the existence and uniqueness of non-linear differential equation by using fixed point results.\\
\\
Consider the initial value problem for the first order non-linear ordinary differential equation:
\begin{equation}\label{12}
\mathfrak{u'(s)}=f(\mathfrak{s,u(s)}), \mathfrak{u(s_{\circ})=r_{\circ}}.
\end{equation}
equation \ref{12} can be rewritten as an integral equation:
\begin{align*}
\mathfrak{u(s)}=\mathfrak{r_{\circ}}+\int_{\mathfrak{s_{\circ}}}^{\mathfrak{s}}f(\mathfrak{t,u(t)})d\mathfrak{t}.
\end{align*}
Let $\mathcal{X}=C([\mathfrak{s_{\circ}-h},\mathfrak{s_{\circ}+h}],\mathbb{R}^{+})$ be the space of continuous function defined on $[\mathfrak{s_{\circ}-h},\mathfrak{s_{\circ}+h}]$ with $(\alpha,\beta)$-metric defined by
\begin{align*}
\mathfrak{G(u_{1},u_{2})}=Sup_{\mathfrak{s}\in[\mathfrak{s_{\circ}-h},\mathfrak{s_{\circ}+h}]} \big|\mathfrak{u_{1}(s)-u_{2}(s)}\big|~\forall ~\mathfrak{u_{1},u_{2}}\in\mathcal{X}.
\end{align*}
Now, Define a mapping $\Psi:\mathcal{X\rightarrow X}$ by
\begin{align*}
\Psi\mathfrak{u(s)}=\mathfrak{r_{\circ}}+\int_{\mathfrak{s_{\circ}}}^{\mathfrak{s}}f(\mathfrak{t,u(t)})d\mathfrak{t}.
\end{align*}
Suppose that $f(\mathfrak{s,u})$ satisfies
\begin{align*}
\big|f(\mathfrak{s,u_{1})}-f(\mathfrak{s,u_{2})}\big|\le \Lambda\big|\mathfrak{u_{1}-u_{2}}\big|,
\end{align*}
for some constant $\Lambda>0~and~\forall~\mathfrak{s}\in[\mathfrak{s_{\circ}-h},\mathfrak{s_{\circ}+h}],~\mathfrak{u_{1},u_{2}}\in\mathbb{R}.$\\
\\
Now for $\mathfrak{u_{1},u_{2}}\in\mathcal{X}$, we have
\begin{align*}
\big|\Psi\mathfrak{u_{1}(s)}-\Psi\mathfrak{u_{2}(s)}\big|&=\bigg|\int_{\mathfrak{s_{\circ}}}^{\mathfrak{s}}f(\mathfrak{t,u_{1}(t))}d\mathfrak{t}-\int_{\mathfrak{s_{\circ}}}^{\mathfrak{s}}f(\mathfrak{t,u_{2}(t))}d\mathfrak{t}\bigg|\\
\\
&=\int_{\mathfrak{s_{\circ}}}^{\mathfrak{s}}\big|f(\mathfrak{t,u_{1}(t))}-f(\mathfrak{t,u_{2}(t))}\big|d\mathfrak{t}\\
\\
&\le\Lambda\int_{\mathfrak{s_{\circ}}}^{\mathfrak{s}}\big|\mathfrak{u_{1}(t)-u_{2}(t)}\big|d\mathfrak{t}.
\end{align*}
Taking Supremum over $\mathfrak{s}\in[\mathfrak{s_{\circ}-h},\mathfrak{s_{\circ}+h}]$, we get
\begin{align*}
Sup_{\mathfrak{s}\in[\mathfrak{s_{\circ}-h},\mathfrak{s_{\circ}+h}]}\big|\Psi\mathfrak{u_{1}(s)}-\Psi\mathfrak{u_{2}(s)}\big|&\le Sup_{\mathfrak{s}\in[\mathfrak{s_{\circ}-h},\mathfrak{s_{\circ}+h}]}\Lambda\int_{\mathfrak{s_{\circ}}}^{\mathfrak{s}}\big|\mathfrak{u_{1}(t)-u_{2}(t)}\big|d\mathfrak{t}.\\
\\
\mathfrak{G}(\Psi\mathfrak{u_{1}},\Psi\mathfrak{u_{2}})&\le\Lambda\mathfrak{hG(u_{1},u_{2})}.
\end{align*}
Thus, $\Psi$ is a contraction. Therefore, by Corollary \ref{22} there exist a unique fixed point which is a solution to the equation \ref{12}.
\section{Conclusions}\label{sec5}
In this paper, we introduced the definition of $(\alpha,\beta)$-metric space and by providing an example we have shown that every $(\alpha,\beta)$-metric space need not be a metric space. Furthermore, we defined some new contraction mappings named $(\alpha,\beta)$-contraction and weak $(\alpha,\beta)$-contraction and explores their role in $(\alpha,\beta)$-metric spaces by presenting fixed point results. Moreover, we proved Kannan's fixed point theorem and Reich's fixed point theorem in the setting of $(\alpha,\beta)$-metric spaces. It focuses on building a theoretical framework for fixed-point theorems and have application in solving various types of integral equations and to establish the existence and uniqueness of a solution of the non-linear differential equations. This work improves our understanding of fixed-point theory and serves as a strong base for future research as the established theorems can be extended to multivalued mappings in the setting of $(\alpha,\beta)$-metric spaces. Furthermore, One can obtain the common fixed point theorems for these novel contraction mappings.

 \end{document}